\documentclass[11pt,twoside]{amsart}
\usepackage{latexsym}
\usepackage{amssymb,amsbsy,amsmath,amsfonts,amssymb,amscd}
\usepackage{graphicx}
\usepackage{hyperref}
\usepackage{xcolor}
\usepackage{dsfont}

\newcommand{\commentout}[1]{}
\newcommand{\RN}[1]{%
  \textup{\uppercase\expandafter{\romannumeral#1}}%
}

\newcommand{\R}{\mathbb{R}}

\newcommand {\e}  {\varepsilon}

\newcommand {\Chi} {{\bf \raise 2pt \hbox{$\chi$}} }

\newcommand {\ep}  {\epsilon}
\newcommand{\beq}{\begin{equation}}
\newcommand{\eeq}{\end{equation}}
\newcommand{\bea} {\begin{array}{rl}}
\newcommand{\eea} {\end{array}}
\newcommand{\bepa}{\left\{ \begin{array}{l}}
\newcommand{\eepa} {\end{array}\right.}
\newtheorem{theorem}{Theorem}[section]
\newtheorem{lemma}[theorem]{Lemma}

\newtheorem{proposition}[theorem]{Proposition}

\def \Ž{\'e}
\def \ˆ{\`a}
\def \{\`e}
\def \{\^e}
\def \{\c{c}}
\def \"{\^{\i}}
\def \™{\^o}
\def \ž{\^u}
\def \‰{\^a}
\def \{\`u}
\def \Š{\"a}
\def\'{\"e}
\def\š{\"o}
\def\•{\"i}
\def\Ÿ{\"u}
\def\…{\"O}
\newcommand {\no}{\noindent}

\title{ {The asymptotics of stochastically perturbed reaction-diffusion equations and front propagation }}

 \author{
Pierre-Louis Lions$^{1,3}$ and Panagiotis E. Souganidis$^{2, 4}$}

\begin{document}

\maketitle
\pagestyle{plain}
\pagenumbering{arabic}

\begin{abstract}

\no  We study the asymptotics of Allen-Cahn-type bistable reaction-diffusion equations which are additively perturbed by a stochastic forcing (time white noise). The conclusion is that the long time, large space behavior of the solutions is governed by an interface moving with curvature dependent normal velocity which is additively perturbed by time white noise.   The result is global in time and does not require any regularity assumptions on the evolving front. The main tools are (i)~the notion of stochastic (pathwise) solution for nonlinear degenerate parabolic equations with multiplicative rough (stochastic) time dependence, which has been developed by the authors, and (ii)~the theory of generalized front propagation put forward by the second author and collaborators   to establish  the onset of moving fronts in the asymptotics of reaction-diffusion equations.

%
\end{abstract}
\medskip

\\
\\
\noindent {\bf AMS Class. Numbers:} 60H15, 35D40.
\bigskip
\section{Introduction}
We investigate the onset of fronts in the long time  and large space asymptotics of bistable reaction-diffusion equations, the prototype being the Allen-Cahn equation, which are additively perturbed by small relatively smooth (mild) stochastic in time forcing. The interfaces evolve with curvature dependent normal velocity which is additively perturbed by time white noise.
No regularity assumptions are made on the regularity of the fronts.  
The results can be  extended to more complicated equations with anisotropic diffusion, drift and reaction which may be periodically oscillatory in space. To keep the ideas  simple, in this note we chose to concentrate on the classical Allen-Cahn equation. 
\smallskip

 We study the behavior, as $\ep\to 0$, of the parabolically rescaled Allen-Cahn equation 
\begin{equation}\label{allencahn}
u^\ep_t -\Delta u^\ep + \dfrac{1}{\ep^2} (f(u^\ep) - \ep \dot B^\ep(t,\omega))=0  \ \text{in} \ \R^d\times (0,\infty) \quad 
u^\ep(\cdot,0)=u^\ep_0,
\end{equation}
where
\begin{equation}\label{nonl}
f\in C^2(\R^d;\R) \ \text{is such that} \begin{cases}\ f(\pm 1)=f(0)=0, f'(\pm 1)>0, \ f'(0)<0 \\[1mm]  f>0 \ \text{in} \ (-1,0), \ f<0 \ \text{in}  \ (0,1), \ \text{and} \ \displaystyle\int_{-1}^{+1} f(u)du=0,
\end{cases}
\end{equation}
that is, $f$ is the derivative of a double well potential with wells of equal depth at, for definiteness, $\pm 1$ and in between maximum at $0$, 
\vskip-.125in
\begin{equation}\label{takis1}
B^\ep (\cdot,\omega) \in C^2([0,\infty);\R) \ \text{is, a.s. in $\omega$,   a mild approximation of the Brownian motion}  \ B(\cdot, \omega),
\end{equation}
that is,  a.s. in $\omega$ and locally uniformly $[0,\infty)$, 
\begin{equation}\label{takis1.1}
 \lim_{\ep \to 0}B^\ep(t, \omega)=B  
 \ \text{and} \  \lim_{\ep \to 0} |\ep \dot B^\ep (t, \omega)|=  \lim_{\ep \to 0} |\ep \ddot B^\ep (t, \omega)| =0, 
\end{equation}
 and there exists an open $\mathcal {O}_0\subset \R^d$ such that 
\begin{equation}\label{takis2}
\begin{cases}
\mathcal {O}_0 =\{x\in \R^d: u^\ep_0(x)>0\},  \ \R^d\setminus \overline {\mathcal {O}_0}=\{x\in \R^d: u^\ep_0(x)<0\}, \ \text{and} \\[1mm] \Gamma_0=\partial \mathcal {O}_0=\partial (\R^d\setminus \overline{ \mathcal {O}_0})=\{x\in \R^d: u^\ep_0(x)=0\}.
\end{cases}
\end{equation}
Although it is not stated explicitly,  it assumed that there exists an underlying probability space, but, for ease of the notation, we omit the dependence on $\omega$ unless it necessary.

\smallskip

We mention here two classical examples of mild approximations. The first is  the convolution $B^\ep(t)=B\star \rho^\ep (t)$, where  $\rho^\ep (t)= \ep^{-\gamma} \rho(\ep^{-\gamma} t)$ with $\rho \in C^\infty$ even and  compactly supported in $(-1,1)$, $\int \rho (t) dt=1$ and $\gamma \in (0,1/2)$. The second is $\dot B^\ep(t)=\ep^{-\gamma} \xi (\ep^{-2\gamma}t)$, where $\xi(t)$  is a stationary, strongly mixing,  mean zero stochastic process such that $\max(|\xi|, |\dot \xi |)\leq M$ and $\gamma \in (0,1/3)$. We refer to Ikeda and Watanabe~\cite{ik} for a discussion. 
\smallskip

Next  we use the notion of stochastic viscosity solutions and the  level set approach to describe  the  generalized evolution (past singularities) of a set with normal velocity
\begin{equation}\label{takis12}
dV=-\text{tr} [Dn] \ dt + d\zeta,
\end{equation}
for some  a continuous path $\zeta \in C([0,\infty);\R)$ with $\zeta(0)=0$. Here  $n$ is the external normal to the front and, hence , $\text{tr}[Dn]$ is the mean curvature. 
\smallskip

Given a triplet   $(\mathcal {O}_0, \Gamma_0, \R^d\setminus \overline {\mathcal {O}_0})$ with $\mathcal {O}_0\subset \R^d$ open,  we say that the sets $(\Gamma_t)_{t>0}$ move with normal velocity  \eqref{takis12}, if, for each $t>0$,  there exists a triplet  $(\mathcal {O}_t, \Gamma_t, \R^d\setminus \overline {\mathcal {O}_t})$, with   $\mathcal {O}_t\subset \R^d$ open, such that 
\begin{equation}\label{takis10}
{\mathcal O}_t =\{x\in \R^d: w(x, t)>0\},  \ \R^d\setminus \overline {\mathcal {O}_t}=\{x\in \R^d: w(x,t)<0\}, \ \text{and} \  \Gamma_t=\{x\in \R^d: w(x,t)=0\},
\end{equation}
where  $w \in \text{BUC}(\R^d\times [0,\infty))$ is the unique stochastic (pathwise) solution
of the level-set initial value pde
\begin{equation}\label{ivp}
dw=\text{tr}(I - \widehat {Dw} \otimes \widehat {Dw} )D^2w  +  |Dw|\circ d\zeta \ \text{in} \ \R^d\times (0,\infty) \quad  w(\cdot,0)=w_0,
\end{equation}
 with $\hat p:=p/|p|$ and $w_0\in \text{BUC}(\R^d)$ is such that 
 \begin{equation}\label{takis11}
\mathcal {O}_0 =\{x\in \R^d: w_0(x)>0\},  \ \R^d\setminus \overline {\mathcal {O}_0}=\{x\in \R^d: w_0(x)<0\}, \ \text{and} \ 
\Gamma_0=\{x\in \R^d: w_0(x)=0\}. 
\end{equation}
Above,  $\zeta$ can be an arbitrary continuous function, in which case ``$\cdot$'' means multiplication.  When, however, $\zeta$ is a Brownian oath, ``$\circ$'' should be interpreted as the classical Stratonovich differential. 
\smallskip

Stochastic viscosity solutions for nonlinear first- and second-order (possibly degenerate) elliptic pde with multiplicative rough time dependence, which include \eqref{ivp}, were introduced and studied in a series of papers of the authors \cite{LS1, LS2, LS3, S}. 
\smallskip

The properties of \eqref{ivp} are used here to adapt the approach  introduced in Evans, Soner and Souganidis~\cite{ess}, Barles, Soner and Souganidis~\cite{bss},  and Barles and Souganidis \cite{bs} to study the onset of moving fronts in the asymptotic limit of reaction-diffusion equations and interacting particle systems with long rage interactions.  This methodology allows to prove global in time asymptotic results and is not restricted to smoothly evolving fronts. 
\smallskip

The main result of the paper is stated next.
\begin{theorem}\label{main}
Assume \eqref{nonl},   \eqref{takis1}, \eqref{takis1.1}, and \eqref{takis2}, and  let $u^\ep$ be the solution  \eqref{allencahn}.
There exists $\alpha_0 \in \R$ such that, 
if $w$ is the solution of \eqref{ivp} with $w_0$ satisfying \eqref{takis11} and $\zeta\equiv \alpha_0 B$, where $B$ is a standard Brownian path, 
then, as $\ep\to 0$,  a.s. in $omega$ and locally uniformly in $(x,t)$,  $u^\ep \to 1$ in $\{(x,t)\in R^d\times (0,\infty): w(x,t)>0\}$ and $u^\ep \to -1$ in $\{(x,t)\in R^d\times (0,\infty): w(x,t)<0\}$,  that is, $u^\ep \to 1$ (resp. $u^\ep \to -1$) inside (resp. outside) a front moving with normal velocity $dV=-\text{tr} [Dn] \ dt + \alpha_0 dB$.
\end{theorem} 

Theorem~\ref{main}, which was already announced in \cite{LS2} is new. It provides a complete characterization of the asymptotic behavior of the  Allen-Cahn equation perturbed by mild approximations of the time white noise. The result holds in all dimensions, it is global in time and does mot require any regularity assumptions on the moving interface. 
\smallskip

In \cite{f1} Funaki  studied the asymptotics of \eqref{allencahn} when $d=2$ assuming  that the initial set is a smooth curve bounding a convex  set. Under these assumptions the evolving set (curve) remains smooth and \eqref{ivp} reduces to a stochastic differential with variable the arc length.  Assuming that the evolving set is smooth, which is true of the initial set is smooth but only for small time, a similar result was announced recently by Alfaro, Antonopoulou, Karali and Matano \cite{aakm}. Assuming convexity at $t=0$, Yip~\cite{y} showed a similar result for all times using  a variational approach. There have also been several other attempts to study to the asymptotics of \eqref{allencahn}  in the graph-like setting and always for small time. 

%
%
%
%
%
%
%
%
\smallskip

Reaction-diffusion equations perturbed additively by white noise arise naturally in the study of hydrodynamic limits of interacting particles in regimes of external magnetization. 
The relationship between the long time, large space behavior of the Allen-Cahn perturbed additively by space-time white noise and fronts moving by additively perturbed mean curvature  was conjectured by Ohta, Jasnow and Kawasaki~\cite{ojk}. Funaki  \cite{f2} obtained results in this direction when $d=1$ where there is no curvature effect.  A recent observation of the authors shows that the general conjecture cannot be mathematically correct. Indeed, it is shown in \cite{LS4} that   the formally conjectured interfaces, which should move by mean curvature additively perturbed with space-time white noise, are not well defined. 

\smallskip

From the phenomelogical point of view, problems like \eqref{allencahn} arise naturally in the phase-field  theory when modeling double-well potentials with depths (stochastically) oscillating in space-time around a common one. This leads to stable equilibria that are only  formally close to $\pm 1$. As a matter of fact, the locations of the equilibria  may diverge due to the strong effect of the white noise. Thus the need to consider approximations of Brownian path in \eqref{allencahn}. 
\smallskip
%

%


The history and literature about the asymptotics of \eqref{allencahn} with or without additive continuous perturbations is rather long. We refer to \cite{bs} for an extensive review as well as references.
\smallskip

The paper is organized as follows. In the next section we discuss in some detail but without proofs  the basic facts about \eqref{ivp} and the notion of generalized front propagation.  In section 3  we prove  Theorem~\ref{main}. 

\smallskip

We note that in the rest of the paper,  instead of  the term stochastic viscosity solution, we will refer to solutions of \eqref{ivp} as pathwise solution. Moreover, $C_0([0,\infty);\R)=\{\zeta \in C([0,\infty);\R) : \zeta(0)=0\}. $ 

\vskip.2in

\section{stochastic viscosity solutions and generalized front propagation}

%
 The main result  about the well-posedness and stability properties of the stochastic viscosity solutions of \eqref{ivp} is stated next. 
 \begin{theorem}\label{takis20}
 For each $w_0\in \text{BUC}(\R^d)$ and $\zeta \in C_0([0,\infty);\R)$ the initial value problem \eqref{ivp} has a unique pathwise solution in $\text{BUC}(\R^d\times [0,\infty))$. Moreover, if $w_n \in \text{BUC}(\R^d\times [0,\infty))$ is the unique solution of \eqref{ivp} for a path $\zeta _n\in C_0([0,\infty);\R)$ and $w_n(\cdot,0)=w_{n,0}$ such that, as $n\to \infty$ and locally uniformly in time and uniformly in space,  $\zeta_n\to \zeta$ and $w_{n,0} \to w_0$,  then, as $n \to \infty$, locally uniformly in time and uniformly in space, $w_n \to w$, the unique solution of \eqref{ivp} with path $\zeta.$
 \end{theorem}
 
 An important tool in the study of evolving fronts is the signed distance function to the front which is defined as 
 \begin{equation}\label{takis22}
 \rho(x,t)=\begin{cases} \rho(x,  {\{y\in \R^d: w(y,t) \leq 0\}}),\\[1mm]
 -\rho(x,  {\{y\in \R^d: w(y,t) \geq 0\}}),
 \end{cases}
 \end{equation}
 where $\rho (x,A)$ is the usual distance between a point $x$ and a set $A$.
 \smallskip
 
 When there is no interior, that is, if 
 $$\partial \{x\in \R^d: w(x,t) < 0\}=\partial \{x\in \R^d: w(x,t) >0\},$$
 then
  \begin{equation*}
 \rho(x,t)=\begin{cases} \rho(x,\Gamma_t) \ \text{if} \  w(x,t) >0,\\[1.3mm]
  - \rho(x,\Gamma_t) \ \text{if} \  w(x,t) <0.
  \end{cases}
  \end{equation*}
  
 The next claim is a direct consequence of the stability properties of the pathwise solutions and the fact that a nondecreasing function of the solution is also solution. When  $\zeta$ is a smooth path, the claim below is established in \cite{bss}. The result for the general path follows by the stability of the pathwise solutions with respect to the local uniform convergence of the paths.
 
 \begin{theorem}\label{takis21}
 Let $w \in \text{BUC}(\R^d\times [0,\infty))$ be the  solution of  \eqref{ivp} and $\rho$ the signed distance function defined by \eqref{takis22}. Then $\underline \rho= \min (\rho, 0)$ and $\overline  \rho=\max(\rho, 0)$ satisfy respectively
 \begin{equation}\label{takis23}
d \underline \rho \leq \text{tr}\left[(I-\dfrac{D \underline\rho\otimes D\underline\rho}{|D\underline \rho|^2}) D^2 \underline \rho \right]dt + |D\underline \rho|\circ d\zeta \ \text{in} \ \R^d\times (0,\infty),
 \end{equation}
 and 
 \beq\label{takis24}
 d \overline \rho \geq \text{tr}\left[(I-\dfrac{D\overline\rho\otimes D\overline\rho}{|D\overline \rho|^2}) D^2 \overline\rho\right] dt + |D\overline\rho|\circ d\zeta  \ \text{in} \  \R^d\times (0,\infty).
\eeq
In addition,
\beq\label{takis25}
 -(D^2\underline \rho D\underline \rho, \underline \rho)\leq 0 \ \ \text{and} \ \  d\underline \rho \leq \Delta \underline \rho + d\zeta  \  \text{in} \ \{\rho <0\}, 
 \eeq
 and
 \beq\label{takis251}
   -(D^2\overline \rho D\overline \rho, \overline \rho)\geq 0 \ \ \text{and} \ \  d\overline \rho \geq \Delta  \overline \rho + d\zeta  \ \text{in} \ \{ \rho >0\}.
 \eeq
%
 \end{theorem}

\section{The asymptotics of the Allen-Cahn equation}

Following the arguments of \cite{bss}, we construct global in time subsolutions and supersolutions of \eqref{allencahn} which do not rely on the regularity of the evolving fronts. Then we use the stability properties of \eqref{ivp} to conclude. 

\smallskip

An important ingredient of the argument is the existence and properties of traveling solutions of \eqref{allencahn} and small additive perturbations of it, which we describe next.

\smallskip

It is well known (see, for example, \cite{bss} for a long list of references) that, if $f$ satisfies \eqref{nonl}, then 
 for every sufficiently small $b$,  there exists a unique strictly increasing traveling wave solution $q=q(x,b)$  and a unique speed $c=c(b)$ of 
\beq\label{takis51}
cq_{\xi}+ q_{\xi\xi}= f(q) -b   \ \text{in} \ \R \quad 
q(\pm \infty, a)=h_{\pm} (b) \quad  q(0,a)=h_0(b),
\eeq
where $h_{-} (b) < h_0(b) <h_{+} (b) )$ are the three solutions of the algebraic equation $f(u) =b$.  Moreover, as $b\to 0$,  
\beq\label{takis51.1}
h_{\pm} (b) \to \pm 1 \ \text{ and} \  h_0(b) \to 0. 
\eeq


%
We summarize the results we need here  in the next lemma. For a sketch of its proof we refer to \cite{bss} and the references there in.  In  what follows, $q_\xi$ and $q_{\xi \xi}$ denote first and second derivatives of $q$  in $\xi$ and $q_b$ the derivative with respect to $b$.

\begin{lemma}\label{takis76}
Assume \eqref{nonl}. There exist $b_0>0, C>0, \lambda >0$ such that, for all $|b|<b_0$, there exist a unique  $c(b)\in \R$, a unique strictly increasing $q(\cdot, b):\R\to \R$ satisfying \eqref{takis51},  \eqref{takis51.1}  and $\alpha_0 \in \R$  such that 
\beq\label{takis77}
0<h_+(b) -q(\xi;b)\leq C e^{-\lambda |\xi|} \  \text{if} \ \xi\geq0 \ \text{and}
\ 0<q(\xi;b)-h_{-}(b) \leq C e^{-\lambda |\xi|} \  \text{if} \ \xi 
\leq 0,
\eeq
\beq\label{takis771}
0<q_\xi (\xi;b) \leq C e^{-\lambda |\xi|}, \ |q_{\xi \xi}(\xi;b)|\leq C e^{-\lambda |\xi|} \ \text{and} \ |q_b| \leq C,
\eeq
\beq\label{takis772}
c(b)=-\frac{h_+(b)-h_-(b)}{\displaystyle \int_{-\infty}^\infty q_\xi(\xi; b)^2 d\xi },  \quad-\alpha_0:= - \frac{d c}{d b}(0)= \frac{2}{\displaystyle \int_{-1}^{1} q_\xi^2 (\xi,0)d\xi,} \quad \text{and} \quad |\dfrac {c(b)}{b} + \alpha_0 | \leq C|b|. 
\eeq
\end{lemma}
\smallskip

In the proof of Theorem~\ref{main} we work with $b=-\ep (\dot B^\ep(t) - a)$ for $a\in (-1,1)$; note that, in view of \eqref{takis1.1}, for $\ep$ sufficiently small, $|b|<b_0$.  To ease the notation, we write 
$$q^\ep(\xi, t, a)=q(\xi, - \ep \dot (B^\ep(t) -  a)) \ \text{and} \ c^\ep(a)=c(-\ep (\dot B^\ep(t) - a)),$$
and we summarize in the next lemma, without a proof, the key properties of $q^\ep$ and $c^\ep$ that we  need later.
\begin{lemma}\label{takis76.1}
Assume \eqref{takis1.1} and the hypotheses of  Lemma~\ref{takis76}.  Then, there exists $C>0$ such that 
\beq\label{takis76.2}
\lim_{\ep\to} \ep |q_t^\ep(\xi, t, a)|=0 \ \text{uniformly on $\xi$ and  $a$ and locally uniformly in  $t\in [0,\infty)$},
\eeq
\beq\label{takis76.3}
\dfrac{1}{\ep} q^\ep_\xi(\xi, t, a) + \dfrac{1}{\ep^2}|q^\ep_{\xi \xi} \xi, t, a)| \leq C \e^{-C\eta/\ep} \ \text{for all $|\xi|\geq \eta$ and all $\eta>0$},
\eeq
\beq\label{takis76.4}
q^\ep_\xi \geq 0 \ \text{and} \ q^\ep_a \geq 0 \ \text{for   all $t\geq 0$ and $\ep, |a|$ sufficiently small, 
} \eeq
and
\beq\label{takis76.5}
|\dfrac{c^\ep}{\ep} + \alpha_0(\ep (\dot B^\ep (t) -a))|=\text{o}(1) \ \text{uniformly for bounded $t$ and $a$}.
\eeq
\end{lemma}

We   prove Theorem~\ref{main} assuming  that $u^\ep_0$ in \eqref{allencahn} is well prepared, that is, has the form 
\beq\label{takis100.1}
u^\ep_0(x)=q^\ep(\dfrac{\rho (x)}{\ep}, 0),
\eeq
where $\rho$ is the signed distance function to $\Gamma_0$ and $q(\cdot, 0)$ is the standing wave solution of \eqref{takis51}.
\smallskip

Going from \eqref{takis100.1} to a general $u^\ep_0$ as in the statement of the theorem is standard in the theory of front propagation. It amounts to   showing  that, in a conveniently 
small time interval, $u^\ep$ can be ``sandwiched'' between functions like the ones in \eqref{takis100.1}. Since this is only technical but standard, we omit the details and we refer to Chen~\cite{c} and \cite{bs} for the details.  
\smallskip

The proof of the result is a refinement of the analogous results of \cite{ess} and \cite{bss}. It is based on using two approximate flows, which evolve with normal velocity $V=-\text{tr}[Dn] + \alpha_0 ( \dot B^\ep(t) - \ep a)$, to construct a subsolution and supesolution \eqref{allencahn}. Since the arguments are similar, here we show the details only for the supersolution construction. 
\smallskip

For fixed $\delta, a>0$ to be chosen below, we consider the solution $w^{a, \delta, \ep}$ of \beq\label{aegean1}
\begin{cases}
w^{a, \delta, \ep}_t = \text{tr}(I - \widehat {Dw^{a, \delta, \ep}} \otimes \widehat {Dw^{a, \delta, \ep}} )D^2w^{a, \delta, \ep}  + \alpha_0(  \dot B^\ep - a) |Dw^{a, \delta, \ep}| \  \text{in} \ \R^d\times (0,\infty), \\[1mm]
 w^{a, \delta, \ep}(\cdot,0)=\rho +\delta,
 \end{cases}
\eeq 

Let  $\rho^{a, \delta, \ep}$ be the signed distance from $\{w^{a, \delta, \ep}=0\}$. It follows from Theorem~\ref{takis21} (see also Theorem~$3.1$ in \cite{bss}) that 
\beq\label{aegean2}
 \rho_t^{a, \delta, \ep} - \Delta \rho^{a, \delta, \ep} {\bf -} \alpha_0( \dot B^\ep -a) \geq 0 \ \text{in} \ \{\rho^{a, \delta, \ep}>0\}.
 \eeq

Following the proof of Lemma~$3.1$ of \cite{ess} we define 
\beq\label{aegean3}
W^{a, \delta, \ep}=\eta_\delta (\rho^{a, \delta, \ep}),
\eeq
where $\eta_\delta:\R\to \R$ is  smooth and such that, for some $C>0$ independent of $\delta$,
\beq\label{aegean4}
\begin{cases}
\eta_\delta \equiv -\delta \ \text{in} \ (-\infty, \delta/4], \quad \eta_\delta \leq -\delta/2   \ \text{in} \ (-\infty, \delta/2], \quad \eta_\delta (z)= z-\delta  \ \text{in} \ [\delta/2, \infty), \ \text{and} \\[1mm]
0\leq \eta_\delta'\leq C \ \text{and} \ |\eta_\delta''|\leq C\delta ^{-1} \ \text{on} \ \R.
\end{cases}
\eeq
Let $T^\star$ be the extinction time of $\{w^{a, \delta, \ep}=0\}$.
A straight forward modification of Lemma~$3.1$ of \cite{ess} leads to the  following claim, which we state without a proof. 
\begin{lemma}\label{aegean4} Assume \eqref{nonl},   \eqref{takis1} and \eqref{takis1.1}. There exists a constant $C>0$, which is independent of $\ep, \delta$ and $a$, such that
\beq\label{aegean5}
W^{a, \delta, \ep}_t - \Delta W^{a, \delta, \ep} - \alpha_0( \dot B^\ep - a) |D W^{a, \delta, \ep}|\geq -\dfrac{C}{\delta} \ \text{in} \ \R^d \times [0, T^\star],
\eeq
\vskip-.1in
\beq\label{aegean6}
W^{a, \delta, \ep}_t - \Delta W^{a, \delta, \ep} - \alpha_0( \dot B^\ep - a) \geq 0 \ \text{in} 
\ \{\rho^{a, \delta, \ep} > \delta /2\},
\eeq
and  
\beq\label{aegean7}
|D W^{a, \delta, \ep}|=1 \ \text{in} 
\ \{\rho^{a, \delta, \ep} >\delta /2\}.
\eeq
\end{lemma}
\smallskip

Finally, we define
\begin{equation}\label{aegean8}
U^{a,  \delta, \ep}(x,t):=q^\ep(\dfrac{W^{a,  \delta, \ep}(x,t)}{\ep}, t, a) \text{ on } \ \R^d\times [0,\infty).
\eeq
\begin{proposition}\label{athens1}
Assume \eqref{nonl}, \eqref{takis1}, \eqref{takis1.1} and \eqref{takis2}. For every $a\in (0,1)$, there exists $\delta_0=\delta_0(a)>0$ such that, for all  $\delta\in (0, \delta_0)$, there exists $\ep_0=\ep_0(\delta,a)>0$, such that, if $\ep\in (0,\ep_0)$, then $U^\ep$ is a supersolution of \eqref{allencahn}.
\end{proposition}
\begin{proof} The arguments are similar to the ones used to  prove the analogous result (Proposition~$10.2$) in  \cite{bss}, hence, here, we only sketch the argument.  Notice that, since we are working at the $\ep>0$ level, we do not have to be concerned about anything stochastic. Below, for simplicity, we argue  as if $w^{a,\delta, \ep}$ had actual derivatives, and we leave it up to the reader to argue in the viscosity sense. Note that $\text{o}(1)$ stands for a  function such that $\lim_{\ep\to 0} \text{o}(1)=0.$ Finally, , throughout the proof, $q^\ep$ and all of its derivatives are evaluated at $(W^{a, \delta, \ep}, t,a),$ a fact which will not be repeated. 

\smallskip
Using the equation satisfied by $q^\ep$, we find
\beq\label{athens100}
\begin{split}
U^{a,  \delta, \ep}_t -\Delta U^{a,  \delta, \ep}  + \dfrac{1}{\ep^2}[f(U^{a,  \delta, \ep} ) - \ep \dot B^\ep(t))=&J^\ep -\dfrac{1}{\ep^2}q^\ep_{\xi \xi}(|DW^{a,  \delta, \ep} |^2 -1)\\[1mm] 
&+\dfrac{1}{\ep}q^\ep_\xi (DW^{a,  \delta, \ep}_t-\Delta W^{a,  \delta, \ep} +\dfrac{c^\ep}{\ep}) + \dfrac{a}{\ep}, 
\end{split}
\eeq
where
\beq\label{athens101}
J^\ep(x,t):=q_b (\dfrac{W^{a,  \delta, \ep} (x,t)}{\ep}, \ep \dot B^\ep(t) -\ep a)\ep \ddot B^\ep(t).
\eeq 
In view of its definition, it is immediate that $|DW^{a,  \delta, \ep}|\leq C$ with $C$ as in \eqref{aegean4}, while \eqref{takis76.2} yields  that, as $\ep \to 0$ and uniformly in $(x,t,
\delta, a)$, 
\beq\label{athens102} J^\ep=\dfrac{\text{o}(1)}{\ep}.\eeq
\smallskip

Next we consider three different cases which depend on the relationship bewteen $\rho^{a, \delta, \ep}$ and $\delta$.
\smallskip

If $\delta /2<\rho^{a, \delta, \ep}<2\delta$, we use \eqref{aegean6}, \eqref{aegean7}, \eqref{takis76.5} and the form of $\eta_\delta$, to rewrite  \eqref{athens100} as 
\beq\label{athens102}
\begin{split}
U^{a,  \delta, \ep}_t -\Delta U^{a,  \delta, \ep}  + \dfrac{1}{\ep^2}[f(U^{a,  \delta, \ep} ) - \ep \dot B^(t))&\geq - \dfrac{1}{\ep}\left[ q^\ep_\xi\right(\frac {c^\ep}{\ep} + \alpha_0 (\ep \dot B^\ep -\ep a) + a  + \text{o}(1),]\\[1mm]
& \geq - \dfrac{1}{\ep}\left[ q^\ep_\xi \text{o}(1)+ a  + \text{o}(1)\right],
\end{split}
\eeq
 It  follows easily that, if $\ep$ and $\delta$ are sufficiently small,  the right side of \eqref{athens102} is positive.

\smallskip
If $\rho^{a,\delta,\ep} \leq \delta/2$, the choise of $\eta_\delta$ implies that
$W^{a, \delta, \ep} \leq -\delta/2.$ 
Hence, \eqref{takis76.3} yields that, for some $C>0$, 
$$\dfrac{1}{\ep} q^\ep_\xi  +\dfrac{1}{\ep^2} |q^\ep_{\xi \xi}| \leq C e^{-C\delta/\ep}.$$

Then using that $|DW^{a,  \delta, \ep}| \leq C$ and \eqref{aegean5} and \eqref{aegean6} in \eqref{athens102}, we get
$$U^{a,  \delta, \ep}_t -\Delta U^{a,  \delta, \ep}  + \dfrac{1}{\ep^2}[f(U^{a,  \delta, \ep} ) + \ep \dot B(t)] \leq -C(\frac{1}{\delta} + 1) e^{-C\delta/\ep} + \text{o}(1) + \dfrac{a}{\ep};$$
notice that, for $\ep$ small enough the right hand side of the inequality above is poistive.
\smallskip

Finally, if $\rho^{a,\delta,\ep} >2\delta$, we use \eqref{aegean6} and \eqref{takis76.3} to conclude as in the previous case.

\end{proof}

We are now ready for the proof of the main result.

\begin{proof}[The proof of Theorem~\ref{main}]
We fix $(x_0,t_0) \in \R^d\times [0,T^\star)$ such that $w(x_0,t_0)=-\beta<0.$ The stability of the pathwise solutions yields that, in the limit $\ep\to 0$, $\delta \to 0$ and $a\to 0$ and uniformly in $(x,t)$, $w^{a,\delta,\ep} \to w$. Thus, we choose sufficiently small $\ep, \delta$ and $a$ so that 
\beq\label{takis110}
w^{a,\delta,\ep} (x_0,t_0) <-\dfrac{\beta}{2}<0.
\eeq 

Then $U^{a,\delta,\ep}$, which is defined in \eqref{aegean7}, is a supersolution of \eqref{allencahn} for sufficiently small $\ep$ and also satisfies, in view of \eqref{takis76.4}, 
$$U^{a,\delta,\ep}(x,0) \geq q^\ep(\dfrac{\rho(x)}{\ep}, 0) \ \text{in} \  \R^d,$$
since 
$$w^{a,\delta,\ep}(x,0)=\eta_\delta(\rho(x) +\delta) \geq \rho(x).$$
The comparison of viscosity solutions of \eqref{allencahn}  then gives 
$$u^\ep \leq U^{a,\delta,\ep} \ \text{in} \ \R^d\times [0,T^\star).$$
We also know, in view of \eqref{takis110}, that $\rho^{a,\delta,\ep}(x_0,t_0)<0$, and, hence,
$$\limsup\limits_{\ep\to 0} u^\ep(x_0,t_0)\leq \limsup\limits_{\ep\to 0} U^{a, \delta, \ep}(x_0,t_0)=-1.$$
For the reverse inequality, we remark that $\hat U(x,t)=-1-\gamma$  is a subsolution of \eqref{allencahn} if $\ep$ and $\gamma>0$ are chosen sufficiently small as it can be seen easily from
$$\hat U_t-\Delta \hat U +\dfrac{1}{\ep^2} (f(\hat U) - \ep \dot B^\ep)\leq C + \dfrac{1}{\ep^2}[-\gamma f'(-1) +\text{o}(1)].$$
The maximum principle then gives, for all $(x,t)$ and sufficiently small $\gamma>0$, 
$$\liminf\limits_{\ep\to 0} u^\ep(x_0,t_0) \geq -1 -\gamma.$$
The conclusion now follows after letting $\gamma \to 0$. 
\smallskip

We remark that a simple modification of the argument above yields the local uniform convergence of $u^\ep$ to $-1$ in compact subsets of $\{w<0\}$.

\end{proof}

\smallskip


\bigskip

%
%


\noindent ($^{1}$) Coll\`{e}ge de France,
11 Place Marcelin Berthelot, 75005 Paris, 
and  
CEREMADE, 
Universit\'e de Paris-Dauphine,
Place du Mar\'echal de Lattre de Tassigny,
75016 Paris, FRANCE\\ 
email: lions@ceremade.dauphine.fr
\\ \\
\noindent ($^{2}$) Department of Mathematics 
University of Chicago, 
5734 S. University Ave.,
Chicago, IL 60637, USA, \ \  
email: souganidis@math.uchicago.edu
\\ \\
($^{3}$) Partially supported by the Air Force Office for Scientific Research  grant FA9550-18-1-0494.
\\ \\ 
($^{4}$)  Partially supported by the National Science Foundation grant DMS-1600129, the Office for Naval Research grant N000141712095 and the Air Force Office for Scientific Research grant FA9550-18-1-0494.

\end{document}